\documentclass[a4paper, 12pt]{article}
\usepackage[a4paper, left=3cm, right=3cm]{geometry}
\usepackage[utf8]{inputenc}
\pdfoutput=1
\usepackage[english]{babel}
\usepackage[T1]{fontenc}
\usepackage{mathtools}
\usepackage{amssymb}
\usepackage{paralist}
\usepackage{tikz}
\usepackage{amsmath}
\usepackage{mathrsfs}
\usepackage{tikz-cd}
\usepackage{amsthm}
\usepackage{lscape}
\usepackage{dynkin-diagrams}
\usepackage{adjustbox}
\usepackage{ytableau}
\usepackage{enumitem}
\usepackage{hyperref}
\usepackage{makecell}
\usepackage{multirow}
\usepackage{csquotes}
\usepackage{orcidlink}
\usetikzlibrary{arrows}
\usepackage{booktabs}
\usepackage{url}

\newtheorem{theorem}{Theorem}[section]
\newtheorem{lemma}[theorem]{Lemma}
\newtheorem{corollary}[theorem]{Corollary}

\newtheorem{proposition}[theorem]{Proposition}
\newtheorem{conjecture}[theorem]{Conjecture}
\theoremstyle{definition}
\newtheorem{example}[theorem]{Example}
\newtheorem{remark}[theorem]{Remark}
\newtheorem{definition}[theorem]{Definition}

\numberwithin{subcase}{case}
\newtheorem*{Ack}{Acknowledgments}
\newsavebox\youngA
\newsavebox\youngB
\newsavebox\youngC
\newsavebox\youngD
\newsavebox\youngE
\newsavebox\youngF

\begin{document}
\title{Orthogonal Determinants of $\mathrm{GL}_n(q)$}
\author{Linda Hoyer\orcidlink{0009-0009-3710-6371}\footnote{linda.hoyer@rwth-aachen.de}}
\date{Lehrstuhl f\"ur Algebra und Zahlentheorie, RWTH Aachen University, Germany}
\maketitle
\begin{abstract}
Let $n$ be a positive integer and $q$ be a power of an odd prime. We provide explicit formulas for calculating the orthogonal determinants $\det(\chi)$, where $\chi \in \mathrm{Irr}(\mathrm{GL}_n(q))$ is an orthogonal character of even degree. Moreover, we show that $\det(\chi)$ is \enquote{odd}. This confirms a special case of a conjecture by Richard Parker. 
\end{abstract}

\section{Introduction}
Let $G$ be a finite group and let $n$ be a positive integer. A representation
$\rho:G \to \mathrm{GL}_n(K)$ for a field $K \subseteq \mathbb{R}$ is also called an \textit{orthogonal} representation: A summation argument shows that there exists a factoring into an orthogonal group \[  \rho:G \to \mathrm{O}(K^n, \beta) \hookrightarrow \mathrm{GL}_n(K), \] where $\beta$ is a non-degenerate symmetric bilinear form. We say that $\rho$ is \textit{orthogonally stable} if there exists an element $d \in K^{\times}$ such that for all such factorings we have $\det(\beta)=d \cdot (\mathbb{Q}(\chi))^2$. In that case, we denote $\det(\rho)=d \cdot (K^{\times})^2$.

Let $\chi$ be a character of $G$. We say that $\chi$ is orthogonal (resp. orthogonally stable) if and only if it is afforded by an orthogonal (resp. orthogonally stable) representation. In \cite{NebeParkerOrtStab} it was shown that for an orthogonally stable character $\chi$, there exists an element $d \in \mathbb{Q}(\chi)^{\times}$ such that for any orthogonally stable representation $\rho:G \to \mathrm{GL}_n(K)$ affording $\chi$ for a field $K/\mathbb{Q}(\chi)$, we have $\det(\rho)=d \cdot (K^{\times})^2$. The square class $\det(\chi):=d \cdot (\mathbb{Q}(\chi)^{\times})^2$ is then called the \textit{orthogonal determinant} of $\chi$.

An important class of orthogonally stable characters is given by the set \[ \mathrm{Irr}^+(G):= \{\chi \in \mathrm{Irr}(G) \mid \chi \ \text{orthogonal}, \ \chi(1) \in 2 \mathbb{Z} \}. \] Indeed, given the character table of $G$ and the orthogonal determinants of the $\mathrm{Irr}^+(G)$-characters, the orthogonal determinant of any orthogonally stable character of $G$ can be calculated. 

The concept of bilinear forms that are invariant under the group action can be generalized to other symmetric algebras with involutions, see also \cite{GeckMonomial}. In the context of this paper, we call these \textit{monomial algebras}. For a monomial algebra $\mathcal{H}$, we may also speak of the $\mathrm{Irr}^+(\mathbb{C}\mathcal{H})$-characters and the orthogonal determinants thereof. An important application are the orthogonal determinants of \textit{Iwahori--Hecke algebras}, which have already been considered in types $A_n$ and $B_n$, see \cite{DeterminantsIwahoriHecke} and \cite{DeterminantsTypeBn}.

A recent long term project aims to calculate the orthogonal determinants of the $\mathrm{Irr}^+(G)$-characters of the finite simple groups, see \cite{OrthAtlas}. The biggest class of the finite simple groups consists of the finite groups of Lie type. Prominent examples of the (not necessarily simple) finite groups of Lie type are given by the groups $\mathrm{GL}_n(q)$, $\mathrm{SL}_n(q)$, $\mathrm{SU}_n(q)$ and $\mathrm{Sp}_{2n}(q)$ for $n$ a positive integer and $q$ a prime power. 

The orthogonal determinants of the groups $\mathrm{SL}_2(q)$, $\mathrm{SL}_3(q)$ and $\mathrm{SU}_3(q)$ have been fully determined, see \cite{BraunNebeSL2} and \cite{HoyerNebe}. The goal of this paper is to extend the results to the groups $\mathrm{GL}_n(q)$ for $q$ odd.

Let $G=\mathrm{GL}_n(q)$ for $n$ a positive integer and $q=p^r$ a power of an odd prime $p$. Let $B$ be a Borel subgroup. An irreducible character of $G$ is called \textit{unipotent} if its restriction to $B$ contains the trivial character. Let now $\chi \in \mathrm{Irr}^+(G)$. There are three possible cases:
\begin{enumerate}[label=(\roman*)]
    \item $\chi$ is \textit{Borel-stable}, i.e., $\mathrm{Res}^G_B(\chi)$ is orthogonally stable.
    \item $\chi$ is unipotent.
    \item The restriction $\mathrm{Res}^G_B(\chi)$ contains a linear character $\theta$ with values strictly contained in $\{-1,1\}$.
\end{enumerate}
 The structure of $B$ as a semidirect product of an abelian group $T$ and a $p$-subgroup $U$ allows for an easy calculation of $\det(\chi)$ for $\chi$ Borel-stable. The final case can be reduced to the case of $\chi$ being unipotent. 
 
 This leaves us with the unipotent characters. Let $\chi \in \mathrm{Irr}^+(G)$ be unipotent; there is a corresponding character $\chi_{\mathcal{H}_q} \in \mathrm{Irr}^+(\mathcal{H}_q)$ of the Iwahori--Hecke algebra $\mathcal{H}_q$ associated to $G$. Condensation techniques now make it possible to reduce the computation of $\det(\chi)$ to that of $\det(\chi_{\mathcal{H}_q})$. Since the orthogonal determinants of the characters of the Iwahori--Hecke algebras of type $A_{n-1}$ are explicitly known, the orthogonal determinants of the characters of $G$ can be fully determined.

The relationship between the orthogonal determinants of the characters of $G$ and the Iwahori--Hecke algebras of type $A_{n-1}$ has further consequences: We show that $\det(\chi)$ is \textit{odd} for all orthogonally stable characters $\chi$ of $G$ if and only if the same holds for the symmetric group $\mathfrak{S}_n$. Since the latter was shown to hold in \cite{HoyerSymmetricParker}, this confirms a conjecture made by Richard Parker in the case of the general linear groups.

The results of this paper have already appeared in this author's PhD thesis \cite{HoyerThesis}.

\begin{Ack}
This paper is a contribution to Project-ID 286237555 – TRR 195 – by the
Deutsche Forschungsgemeinschaft (DFG, German Research Foundation).
\end{Ack}

\section{Monomial algebras}
In this section, we introduce \textit{monomial algebras}, that is, symmetric algebras with a \enquote{nice} (anti)involution. Monomial algebras have already appeared in \cite{GeckMonomial} where Frobenius--Schur indicators of these algebras were introduced, although we do not assume our algebras to be split. Important examples of monomial algebras include group algebras and Iwahori--Hecke algebras in characteristic $0$. The main motivation to consider these algebras is to have a common language for the \textit{orthogonal determinants} appearing throughout this paper.

Let $\mathcal{H}$ be a finite-dimensional associative semi-simple $\mathbb{Q}$-algebra. In particular, the field $\mathbb{C}$ is a splitting field of $\mathcal{H}$; we denote $K\mathcal{H}:=K \otimes_{\mathbb{Q}} \mathcal{H}$ for a subfield $K$ of the complex numbers. We assume all of the modules considered to be finite-dimensional left modules. For $K\subseteq \mathbb{C}$ a field, any $K\mathcal{H}$-module $V$ gives rise to a representation $\rho:K\mathcal{H} \to \mathrm{End}(V)$. The \textit{character} of $V$ is defined by \[ \chi_V: K\mathcal{H} \to K, \ \chi_V(h):=\mathrm{trace}(\rho(h)), \ h \in K\mathcal{H}. \] The set of \textit{irreducible characters} is given by \[ \mathrm{Irr}(\mathbb{C}\mathcal{H}):= \{\chi_V \mid V \ \text{simple} \  \mathbb{C}\mathcal{H} \text{-module}  \}.  \]
For $\chi=\chi_V$ a character of $K\mathcal{H}$, we set $\deg(\chi):=\dim(V)$ to be the \textit{degree} of $\chi$.

\begin{definition}
An involution on $\mathcal{H}$ is a $\mathbb{Q}$-linear map $\dagger: H \to H$ such that $(h^{\dagger})^{\dagger}=h$ for all $h \in \mathcal{H}$ and $(hh')^{\dagger}=(h')^{\dagger}h^{\dagger}$ for all $h,h' \in \mathcal{H}$.
\end{definition}
In other words, an involution is an isomorphism to its opposite algebra which is of order $2$.
  \begin{definition}
A trace function on $\mathcal{H}$ is a $\mathbb{Q}$-linear map $\tau:\mathcal{H} \to \mathbb{Q}$ such that $\tau(hh')=\tau(h'h)$ for all $h, h' \in \mathcal{H}$. We say that the pair $(\mathcal{H},\tau)$ is a symmetric algebra if the symmetric bilinear form \[ \mathcal{H} \times \mathcal{H} \to \mathbb{Q}, (h,h') \mapsto \tau(hh') \] is non-degenerate. We then call $\tau$ a symmetrizing trace for $\mathcal{H}$.
    \end{definition}
For more information about symmetric algebra, we refer to \cite{GeckPfeifferCoxeterIwahori}.
    \begin{definition}
    \label{DefinitionMonomial}
      A symmetric algebra $(\mathcal{H},\tau)$ is called a monomial algebra if the following are satisfied:
            \begin{enumerate}[label=(\roman*)]
                \item There is an involution $\dagger: \mathcal{H} \to \mathcal{H}$ such that $\tau(h^{\dagger})=\tau(h)$ for all $h \in \mathcal{H}$.
                \item There is a finite set $W$ and a $\mathbb{Q}$-basis $\{b_w \mid w \in W\}$ of $\mathcal{H}$ indexed by $W$ such that there is a bijective map $\iota: W \to W$ with $b_w^{\dagger}=b_{\iota(w)}$ and $$\tau(b_w b_{w'}^{\dagger})= 
                \begin{cases}
                    a_w & \text{if} \ \iota(w)=w', \\
                    0 & \text{if} \ \iota(w) \neq w',
                \end{cases} $$ where $a_w=a_{\iota(w)} \in \mathbb{Q}^{\times}$.
            \end{enumerate}
    \end{definition}
    \begin{example}
    \label{ExampleGroupAlgebraMonomial}
Let $G$ be a finite group. We claim that the group algebra $\mathcal{H}:=\mathbb{Q}G$ has a natural structure as a monomial algebra. We let $\tau:\mathcal{H} \to \mathbb{Q}$ be such that \[ \tau(g):= \begin{cases}
1, & \ \text{if} \ g=1, \\
0, & \ \text{if} \ g \neq 1.
\end{cases} \] for $g \in G$. The set $G$ is a canonical choice for a basis of $\mathcal{H}$, we further set $\iota(g)=g^{\dagger}:=g^{-1}$. We have $a_g=1$ for all $g \in G$. The set of irreducible characters of $\mathbb{C}G$ will be denoted by $\mathrm{Irr}(G)$.     
    \end{example}

Assume from now on that $\mathcal{H}$ is a monomial algebra.
\begin{definition}
 Let $\chi$ be a character of $\mathbb{C}\mathcal{H}$. We set \[ \mathbb{Q}(\chi)=\mathbb{Q}( \{ \chi(b_w) \mid w \in W \}  ) \] to be the \textit{character field} of $\chi$.
\end{definition}

\begin{definition}
Let $K \subseteq \mathbb{R}$ be a field, $V$ be a $K\mathcal{H}$-module and $\beta: V \times V \to K$ be a non-degenerate symmetric bilinear form on $V$. We call the pair $(V, \beta)$ an orthogonal $K\mathcal{H}$-module, if \[ \beta(hv,v')=\beta(v, h^{\dagger} v') \] for all $v, v' \in V$, $h \in H$. Furthermore, the character $\chi_V$ is called an orthogonal character. We define \[ \mathrm{Irr}^+(\mathbb{C}\mathcal{H}):=\{\chi \in  \mathrm{Irr}(\mathbb{C}\mathcal{H}) \mid \chi \ \text{orthogonal}, \ \mathrm{deg}(\chi) \in 2 \mathbb{Z} \}. \] 
\end{definition}

For $\mathcal{H}=\mathbb{Q}G$ a group algebra of a finite group, we may write $\mathrm{Irr}^+(G)$ instead of $\mathrm{Irr}^+(\mathbb{C}G)$.
\begin{example}
   Note that $(\mathcal{H},\beta_{\tau})$ with the bilinear form $\beta_{\tau}(h,h'):=\tau( h  (h')^{\dagger})$ for $h, h' \in \mathcal{H}$ is itself an orthogonal $\mathcal{H}$-module. There is a canonical orthogonal basis given by the standard basis $\{b_w \mid w \in W\}$.
\end{example}

\subsection{Orthogonal determinants of monomial algebras}
In \cite{NebeOrtDet}, the orthogonal determinants of the characters in $\mathrm{Irr}^+(\mathbb{C}G)$ for $G$ a finite group were introduced. We will state the main results appearing in loc. cit., which apply with virtually no changes to our more general situation.

Let $V$ be a finite-dimensional $K$-vector space for a field  $K \subseteq \mathbb{R}$ and let $\beta$ be a symmetric non-degenerate bilinear form on $V$. Recall that the \textit{determinant} of $\beta$ is the square class of the determinant of the Gram matrix with respect to an arbitrary basis $\{v_i \mid 1 \leq i \leq n \}$ of $V$, i.e., \[ \det(\beta):=\det(\beta(v_i, v_j)_{1 \leq i,j \leq n}) \cdot (K^{\times})^2.\]

\begin{theorem}
\label{MainThmOrthDetMonomial}
Let $\mathcal{H}$ be a monomial algebra with involution $\dagger:\mathcal{H} \to \mathcal{H}$. Let $\chi \in \mathrm{Irr}^+(\mathbb{C}\mathcal{H})$ be of degree $n$.
\begin{enumerate}[label=(\roman*)]
    \item There is a unique square class $\det(\chi):=d \cdot (\mathbb{Q}(\chi)^{\times})^2 \in \mathbb{Q}(\chi)^{\times}/(\mathbb{Q}(\chi)^{\times})^2$ such that for any orthogonal $K\mathcal{H}$-module $(V,\beta)$ affording $\chi$ for any field $K \subseteq \mathbb{R}$, we have that \[ \det(\beta)=d \cdot (K^{\times})^2. \] We call $\det(\chi)$ the orthogonal determinant of $\chi$.    
    \item Let $\rho:K\mathcal{H} \to K^{n \times n}$ for a field $K \subseteq \mathbb{R}$ be a representation affording the character $\chi$. There exists an element $h \in \mathcal{H}$ with $h^{\dagger}=-h$ and $\det(\rho(h)) \neq 0$. For any such element $h$, we have that \[ \det(\chi)=\det(\rho(h)) \cdot (\mathbb{Q}(\chi)^{\times})^2. \]
\end{enumerate}
\end{theorem}
\begin{remark}
Let $G$ be a finite group and let $\mathcal{H}=\mathbb{Q}G$. In \cite{NebeParkerOrtStab}, it is shown that Theorem \ref{MainThmOrthDetMonomial} still holds if we only assume $\chi$ to be orthogonally stable.   
\end{remark}

\section{Orthogonal determinants of finite groups}

Let $G$ be a finite group. The orthogonal determinants of $G$ are special in two regards. On the one hand, orthogonal determinants are also defined for more general orthogonally stable characters, in contrast to solely the $\mathrm{Irr}^+(\mathbb{C}\mathcal{H})$-characters for $\mathcal{H}$ a monomial algebra. On the other hand, we may use techniques from the representation theory of finite groups, such as restriction and induction from subgroups, to calculate the orthogonal determinants of $G$.

\begin{remark}(see \cite[Proposition 5.7]{NebeParkerOrtStab})
\label{DetOrthogonallySimple}
Let $\rho$ be an irreducible orthogonally stable representation of $G$ and let $\chi$ be the character of $\rho$. There are three cases to consider.
\begin{enumerate}[label=(\roman*)]
\item $\chi=\psi+ \overline{\psi}$, where $\psi \in \mathrm{Irr}(G)$ has Frobenius--Schur indicator $0$, i.e., $\psi$ achieves non-real values. Then $\det(\chi)=\delta^{\psi(1)} \cdot (\mathbb{Q}(\chi))^2$, where $\delta \in \mathbb{Q}(\chi)$ is chosen such that \[ \mathbb{Q}(\chi)[\sqrt{-\delta}]=\mathbb{Q}(\psi) .\]
    \item $\chi=2\psi$, where $\psi \in \mathrm{Irr}(G)$ has Frobenius--Schur indicator $-1$, i.e., $\psi$ has real values but is not afforded by a real representation. Then $\det(\chi)=1 \cdot (\mathbb{Q}(\chi))^2$.
    \item $\chi \in \mathrm{Irr}^+(G)$. 
\end{enumerate}
\end{remark}
If $\chi$ is the sum of orthogonally stable characters, then $\det(\chi)$ can be deduced from the orthogonal determinants of its orthogonally simple components, see also \cite[Proposition 5.17]{NebeParkerOrtStab}. These considerations lead to a complete answer for the orthogonal determinants of some specific classes of finite groups.
\begin{remark}
\label{Abelian Group p-Group Determinant}
 \begin{enumerate}[label=(\roman*)]
\item Let $G$ be an abelian group and let $\chi=\sum_{i=1}^m \theta_i$ be an orthogonally stable character of $G$, with $\theta_i \in \mathrm{Irr}(G)$ for all $i$. Recall that the irreducible characters of abelian groups are given by homomorphisms $\theta_i:G \to \mathbb{C}^{\times}$. We conclude that $\chi$ is orthogonally stable precisely when the Frobenius--Schur indicator of each of the $\theta_i$ is equal to $0$ and each of the irreducible components comes with a complex conjugate partner, i.e., $\chi=\sum_{i=1}^{1/2 m} (\theta'_i+\overline{\theta'_i})$ for some $\theta'_i \in \mathrm{Irr}(G)$. The beforementioned formulas now allow to calculate $\det(\chi)$ in all cases.
    \item Let $p$ be an odd prime and assume $G$ is a $p$-group. Let $\chi$ be an orthogonally stable character of $G$. In \cite[Theorem 4.3]{NebeOrthDisc} an explicit formula for $\det(\chi)$ is given. We only state the formula in a special case:
    
    Assume further that $\mathbb{Q}(\chi)=\mathbb{Q}$. Let $q=p^r$ for some positive integer $r$ such that $q-1 \mid \chi(1)$. Then \[ \det(\chi)=q^{\chi(1)/(q-1)} \cdot (\mathbb{Q}^{\times})^2. \]
\end{enumerate}   
\end{remark}

The following two statements address the relationship between the orthogonal determinants of $G$ and those of its subgroups. We fix a subgroup $H$ of $G$.
\begin{lemma}
\label{OrthDetRestriction}
Let $\chi$ be an orthogonally stable character of $G$. Assume that $\mathrm{Res}^G_H(\chi)$ is an orthogonally stable character of $H$. Then $\det(\chi)=\det(\mathrm{Res}^G_H(\chi)) \cdot (\mathbb{Q}(\chi)^{\times})^2$. 
\end{lemma}

\begin{lemma}
 \label{Determinant Induced Character}
Let $\psi$ be an orthogonal character of $H$ such that $\chi=\mathrm{Ind}^G_H(\psi)$ is orthogonally stable with $\mathbb{Q}(\psi)=\mathbb{Q}(\chi)$. Then 
\[ \det(\chi)=\begin{cases} 1 \cdot (\mathbb{Q}(\chi)^{\times})^2,  &\ \text{if} \ [G:H] \ \text{even}, \\
                        \det(\psi),  &\ \text{if} \ [G:H] \ \text{odd}.
\end{cases} \] 
\end{lemma}
\begin{proof}
Let $(V,\beta)$ be an orthogonal $KH$-module affording the character $\psi$ for a field $K/\mathbb{Q}(\chi)$. Let $x_1, \dots, x_N$ be representatives of the cosets $G/H$ and let \[ W=\bigoplus_{i=1}^n x_iV \] be the induced representation. Define the bilinear form $\beta_W$ on $W$ by \[ \beta_W(x_i v, x_jv')=\delta_{ij} \beta(v,v'), \ v,v' \in V \] where $\delta_{ij}$ depicts the Kronecker delta. This makes $(W, \beta_W)$ into an orthogonal $KG$-module. We verify that $\det(\beta_W)=\det(\beta)^{N}$ from which the statement follows. 
\end{proof}    
Recall that if $G_1, G_2$ are finite groups, then \[  \mathrm{Irr}(G_1 \times G_2)= \{\chi_1 \boxtimes \chi_2 \mid \chi_1 \in \mathrm{Irr}(G_1), \chi_2 \in \mathrm{Irr}(G_2) \}, \] where $(\chi_1 \boxtimes \chi_2)(g_1,g_2):=\chi_1(g_1)\chi_2(g_2)$ for $g_1 \in G_1$, $g_2 \in G_2$ is the outer product of the characters. 
    \begin{lemma}
    \label{Determinants Direct Product Groups}
    Let $G_1, G_2$ be finite groups, $\chi_1 \in \mathrm{Irr}(G_1), \chi_2 \in \mathrm{Irr}(G_2) $. Let $G=G_1 \times G_2$ and $\chi=\chi_1 \boxtimes \chi_2$. Then $\chi \in \mathrm{Irr}^+(G)$ if and only if one of the following holds:
     \begin{enumerate}[label=(\roman*)]
     \item We have $\iota(\chi_1)=\iota(\chi_2)=1$ and at least one of the characters has even degree. Assume without loss of generality that $\chi_1(1)$ is even. Then $$\det(\chi) =\det(\chi_1)^{\chi_2(1)} \cdot (\mathbb{Q}(\chi)^{\times})^2.$$
     \item We have $\iota(\chi_1)=\iota(\chi_2)=-1$ and at least one of the characters has even degree. Then $$\det(\chi) =1\cdot (\mathbb{Q}(\chi)^{\times})^2.$$
     \end{enumerate}
    \end{lemma}
    \begin{proof}
It is clear by definition that the Frobenius--Schur indicator is multiplicative. Assume that $\chi \in \mathrm{Irr}^+(G)$. In particular, its degree is even and we can assume without loss of generality that $\chi_1$ has even degree. The statement now follows since $\mathrm{Res}^G_{G_1 \times \{1\}}(\chi)=\chi_2(1)\chi_1$ is orthogonally stable.
    \end{proof}

\subsection{Condensation}
\label{SectionDeterminantsIwahoriTypeA}
The goal of this subsection is to present a generalized version of Lemma \ref{OrthDetRestriction}, extending it to cases where the restriction to a subgroup is not necessarily orthogonally stable but may include positive even multiples of the trivial character. The idea is to consider the corresponding \textit{Hecke algebra} of the subgroup and the orthogonal determinant of the associated \textit{condensed module}. Condensation techniques are a well-established tool in computational group theory, see for instance \cite{Condensation}. We also refer to  \cite[Section 3.3.2]{OrthAtlas} for an application to orthogonal determinants. The main source of this subsection is \cite[§11D]{CurtisReinerVol1}.

We fix a finite group $G$ and a subgroup $B$. We let $\mathcal{H}:=\mathbb{Q}G$ be the monomial algebra as in Example \ref{ExampleGroupAlgebraMonomial} with symmetrizing trace $\tau$ and involution $\dagger$.
\begin{definition}
 \label{DefinitionHeckeAlgebra}
Consider the idempotent \[  e_B:=\frac{1}{|B|} \sum_{h \in B} h \in \mathcal{H}. \] The Hecke algebra $\mathcal{H}_B$ is defined to be the subalgebra $\mathcal{H}_B:=e_B \mathcal{H} e_B \subseteq \mathcal{H}$.
\end{definition}

\begin{proposition}(see \cite[Proposition 11.34]{CurtisReinerVol1})
Let $D(G,B)= \{x_1, \dots, x_N \} \subseteq G$ be a set of representatives of the double cosets $B \backslash G/B$. We set $D_{i}:=Bx_iB$ for any $x_i \in D(G,B)$ and define \[ T_{x_i}:=\frac{1}{|B|} \sum_{h \in D_i} h \in \mathcal{H}. \] The set $\{T_x \mid x \in D(G,B)   \}$ then is a $\mathbb{Q}$-basis of $\mathcal{H}_B$, called the Schur basis.
\end{proposition}
\begin{lemma}(see \cite[§11D]{CurtisReinerVol1})
The restriction of $\tau$ and $\dagger$ from $\mathcal{H}$ to $\mathcal{H}_B$ and the choice of the Schur basis gives $\mathcal{H}_B$ the natural structure of a monomial algebra. It holds that $a_x=|BxB/B|$ for $x \in D(G,B)$.
\end{lemma}

We now construct the irreducible characters of $\mathcal{H}_B$.
\begin{definition}
\label{StableModule}
Let $\chi \in \mathrm{Irr}(G)$ and let $V$ be a $KG$-module affording $\chi$ for some field $K \subseteq \mathbb{C}$. Define the condensed module \[ V_B:=\mathrm{Stab}_{KB}(V)=  \{ v \in V \mid bv=v \ \text{for all} \ b \in B  \}; \] it is a $K\mathcal{H}_B$-module with $\dim(V_B)=\langle \mathrm{Res}^G_B(\chi),\mathbf{1}_B \rangle$. If $\dim(V_B) \neq 0$, denote $\chi_{B}:=\chi_{V_B}$ to be the corresponding character of $K\mathcal{H}_B$.
\end{definition}

\begin{proposition}
\label{1to1Correspondence}
There is a $1$-to-$1$ correspondence between the set \[ I:=\{ \mathrm{Irr}(G) \mid \langle \mathrm{Res}^G_B(\chi),\mathbf{1}_B \rangle \neq 0 \} \]  and $\mathrm{Irr}(\mathbb{C}\mathcal{H}_B)$. Moreover, if $\chi \in I$ is orthogonal, then so is $\chi_B$.
\end{proposition}
\begin{proof}
The first statement is \cite[Theorem 11.25]{CurtisReinerVol1}. For the second statement, let $(V, \beta)$ be an orthogonal module affording the character $\chi$. Then clearly $(V_B, \beta_{|V_B})$ is also an orthogonal module.    
\end{proof}

We can now state the following main theorem.
\begin{theorem}
\label{Condensation_OrthDet}
Let $\chi \in \mathrm{Irr}^+(G)$ such that $\langle \mathrm{Res}^G_B(\chi),\mathbf{1}_B) \rangle \neq 0$. Assume that $\deg(\chi_{\mathcal{H}_B})$ is even and that  \[ \psi:=\mathrm{Res}^G_B(\chi)-\deg(\chi_{\mathcal{H}_B}) \mathbf{1}_B \] is an orthogonally stable character of $B$. Then $\chi_{\mathcal{H}_B} \in \mathrm{Irr}^+(\mathbb{C}\mathcal{H}_B)$ and \[ \det(\chi)=\det(\chi_{\mathcal{H}_B}) \det(\psi) \cdot (\mathbb{Q}(\chi)^{\times})^2. \]
\end{theorem} 
\begin{proof}
By Proposition \ref{1to1Correspondence}, the character $\chi_{\mathcal{H}_B}$ is orthogonal, which implies that $\chi_{\mathcal{H}_B} \in \mathrm{Irr}^+(\mathbb{C}\mathcal{H}_B)$.

Let $V$ be a $KG$-module affording the character $\chi$ for a field $K \subseteq \mathbb{R}$. Let $U \subseteq V$ be the $KB$-submodule affording the character $\psi$; we arrive at a direct sum $V=V_B \oplus U$. Let $n:=\dim(V)$ and $k:=\dim(V_B)$. We choose a basis $(v_1, \dots, v_n)$ of $V$ such that $(v_1, \dots, v_{k})$ and $(v_{k+1}, \dots, v_n)$ are bases of $V_B$ and $U$, respectively.

With respect to these bases, we arrive at three representations of the corresponding algebras:
\begin{enumerate}[label=(\roman*)]
\item $\rho:KG \to K^{n \times n}$ affording the character $\rho$ of $G$,
\item $\phi:K\mathcal{H}_B \to K^{k \times k}$ affording the character $\chi_{\mathcal{H}_B}$ of $\mathcal{H}_B$,
\item $\pi:KB \to K^{(n-k) \times (n-k)}$ affording the character $\psi$ of $KB$.
\end{enumerate}
Let $\dagger$ be the involution of $\mathcal{H}$. By Theorem \ref{MainThmOrthDetMonomial}, there are elements $h_1 \in \mathcal{H}_B$, $h_2 \in \mathbb{Q}B$ such that $h_i^{\dagger}=-h_i$ for $i=1,2$ and both $\det(\phi(h_1))$ and $\det(\pi(h_2))$ are non-zero. Let $h:=h_1+h_2 \in \mathcal{H}$. Then $h^{\dagger}=-h$ and $\rho(h)=\mathrm{diag}(\phi(h_1),   \pi(h_2)   )$. Another application of Theorem  \ref{MainThmOrthDetMonomial} leads us to \[ \det(\chi)=\det(\chi_{\mathcal{H}_B}) \det(\psi) \cdot (\mathbb{Q}(\chi)^{\times})^2, \] which we wanted to show.
\end{proof}

\section{Orthogonal determinants of Iwahori--Hecke algebras of \texorpdfstring{type $A_{n-1}$}{type An-1}}
\label{IwahoriHecke}
The Iwahori--Hecke algebras of type $A_{n-1}$ occur as the Hecke algebras of the general linear groups $\mathrm{GL}_n(q)$ over a finite field with $q$ elements with respect to a Borel subgroup $B$. They are deformations of the group algebras $\mathbb{Q}\mathfrak{S}_n$ of the symmetric groups on $n$ letters and as such inherit many of their properties. The orthogonal determinants of these algebras have been completely determined, see \cite[Theorem 4.11]{DeterminantsIwahoriHecke}. The formulas were further generalized and simplified in \cite{DeterminantsTypeBn} for the Iwahori--Hecke algebras of type $B_n$, leading to the formulas presented in this section.

Let $p$ be a prime and $q$ be a power of $p$. Let $n$ be a postive integer. Throughout this section, we let $G=\mathrm{GL}_n(q)$ be the general linear group over a field with $q$ elements. Let $B \subseteq G$ be the Borel subgroup of upper triangular matrices; we let $\mathcal{H}_q=e_B \mathbb{Q}G E_B$ be the corresponding Hecke algebra. The field $\mathbb{Q}$ is a splitting field of $\mathcal{H}_q$, so we may write $\mathrm{Irr}(\mathcal{H}_q)$ for the set of its irreducible characters.

We identify the subgroup of permutation matrices, i.e., the matrices with entries $\{0,1\}$ with exactly one $1$ in every row and column, with the symmetric group $\mathfrak{S}_n$. Let $S=\{s_1, \dots, s_{n-1} \} \subseteq \mathfrak{S}_n$ be the set of simple transpositions.

\begin{proposition}(see \cite[Theorem 8.4.6]{GeckPfeifferCoxeterIwahori})
The algebra $\mathcal{H}_q$ has the Schur basis $\{T_w \mid w \in  \mathfrak{S}_n \}$ with relations \[ T_s T_{w}= \begin{cases}
T_{sw}, \ & \text{if} \ \ell(sw)=\ell(w)+1, \\
qT_{sw}+(q-1)T_w, \ & \text{if} \ \ell(sw)=\ell(w)-1,
\end{cases} \]
for any $s \in S, w \in \mathfrak{S}_n$.
\end{proposition}

A \textit{partition} of $n$ is a sequence $\lambda=(a_1, \ldots, a_m)$ of positive integers such that $a_1 \geq a_2 \geq \ldots \geq a_m$ and $\sum_{i=1}^m(a_i)=n$. We may also write $\lambda \vdash n$. Recall that the irreducible characters of the symmetric groups are labeled by partitions; we write $\chi_{\lambda}^{(1)} \in \mathrm{Irr}(\mathfrak{S}_n)$ for the character corresponding to $\lambda \vdash n$.
\begin{definition}
\label{DefinitionUnipotent}
A character $\chi \in \mathrm{Irr}(G)$ is called unipotent if and only if \[ \langle \mathrm{Res}^G_B(\chi),\mathbf{1}_B \rangle \neq 0. \] There is a $1$-to-$1$ correspondence between the unipotent characters of $G$ and the irreducible characters of $\mathfrak{S}_n$; we write $\chi_{\lambda}^{(q)}$ for the unipotent character corresponding to a partition $\lambda \vdash n$.
The associated character of $\mathcal{H}_q$ is denoted by $\widetilde{\chi}_{\lambda}^{(q)}$. 
\end{definition}
Note that the unipotent characters are afforded by representations over $\mathbb{Q}$, so all unipotent characters are orthogonal. Moreover, it is clear that $\deg(\widetilde{\chi}_{\lambda}^{(q)})=\deg(\chi_{\lambda}^{(1)})$.

Recall that the \textit{Young diagram} of a partition $\lambda=(a_1, \ldots, a_m)$ of $n$ is an arrangement of boxes with $a_i$ boxes in the i-th row. 

\begin{definition} Let $\lambda \vdash n$.
\begin{enumerate}[label=(\roman*)]
    \item A \textit{Young tableau} of $\lambda$ is the bijective filling of the boxes of the Young diagram of $\lambda$ with the entries $1, \ldots, n$. 
    \item We call a Young tableau \textit{standard} if the entries in the boxes read from left to right and from top to bottom are always increasing. We denote by $T_{\lambda}$ the set of all standard Young tableaux of $\lambda$.
    \item We set $t_{\lambda} \in T_{\lambda}$ to be the standard Young tableau with entries $1, \ldots, a_1$ in its first row, $a_1+1, \ldots,a_1+a_2$ in its second row, and so on.
\end{enumerate}
\end{definition}
\begin{figure}[h]
    \centering
    \begin{ytableau}
                1 & 2 & 3 & 4\\
                5 & 6 & 7 \\
                8 & 9
            \end{ytableau} 
    \caption{The standard Young tableau $t_{\lambda}$ for $\lambda=(4,3,2)$.}
    \label{Figure1}
\end{figure}

It is well known that $\deg(\chi_{\lambda}^{(1)})=|T_{\lambda}|$. Note that the group $\mathfrak{S}_n$ acts regularly on the set of Young tableaux for $\lambda \vdash n$ by permuting the entries of a given tableau.

\begin{definition}
\begin{enumerate}[label=(\roman*)]
    \item Let $u, w \in \mathfrak{S}_n$. We say that $u \leq_L w$ in the weak left partial order if \[ \ell(w)=\ell(u)+\ell(wu^{-1}).\]
    \item Let $\lambda=(a_1, \ldots, a_m) \vdash n$ and let $t_1, t_2 \in T_{\lambda}$. There are unique elements $w_1, w_2 \in \mathfrak{S}_n$ such that $t_i=w_i \cdot t_{\lambda}$ for $i=1,2$. We put a partial order on $T_{\lambda}$ by saying that $t_1 \leq t_2$ if and only if $w_1 \leq_L w_2$.
\end{enumerate}
\end{definition}

For $k$ a positive integer, we define the polynomials \[[k]_x:=\frac{x^k-1}{x-1}=1+x+x^2+ \ldots + x^{k-1}. \] 
\begin{definition}
 Let $\lambda$ be a partition of $n$. We inductively define polynomials $a_{t} \in \mathbb{Q}[x]$ for $t \in T_{\lambda}$.
 \begin{enumerate}[label=(\roman*)]
     \item We set $a_{t_{\lambda}}:=1. $ 
     \item Let $t, t' \in T_{\lambda}$ with $t < t'$ and assume there is a simple transposition $s=(k,k+1) \in S$ such that $s \cdot t=t'$. Let $(i,j)$ (resp. $(\ell,m)$) be the positions of $k$ (resp. $k+1$) in $t$. Set $c:=(j-i)-(m-\ell)-1$. Assume we already know $a_t$. Then we define the polynomial $a_{t'}$ by
 \[ a_{t'} :=x[c+2]_x[c]_x a_{t}.\] 
 \end{enumerate}
 These polynomials are well-defined and do not depend on the choice of $t$. We further define \[ f_{\lambda}:=\prod_{t \in T_{\lambda}} a_t. \]
\end{definition}

We can now state the main result of this section.
\begin{theorem}(cf. \cite[Section 3]{DeterminantsTypeBn})
\label{OrthogonalDeterminantsHeckeTypeAn}
Let $\lambda \vdash n$  such that $\widetilde{\chi}_{\lambda}^{(q)} \in \mathrm{Irr}^+(\mathcal{H}_q)$. Then \begin{align*}
 \det(\chi_{\lambda}^{(1)})&=f_{\lambda}(1) \cdot (\mathbb{Q}^{\times})^2, \\
 \det(\widetilde{\chi}_{\lambda}^{(q)})&=f_{\lambda}(q) \cdot (\mathbb{Q}^{\times})^2. \\
\end{align*}
\end{theorem}   
    
\savebox\youngA{\begin{ytableau}
    1 & 2 & 3 \\ 4 \\ 5
    \end{ytableau}}%
\savebox\youngB{\begin{ytableau}
    1 & 2 & 4 \\ 3 \\ 5
    \end{ytableau}}%
    \savebox\youngC{\begin{ytableau}
   1 & 3 & 4 \\ 2 \\ 5
    \end{ytableau}}%
\savebox\youngD{\begin{ytableau}
    1 & 2 & 5 \\ 3 \\ 4
    \end{ytableau}}%
    \savebox\youngE{\begin{ytableau}
    1 & 3 & 5 \\ 2 \\ 4
    \end{ytableau}}%
\savebox\youngF{\begin{ytableau}
    1 & 4 & 5 \\ 2 \\ 3
    \end{ytableau}}%

    \begin{figure}[ht]
\adjustbox{scale=0.8,center}{
$$\begin{tikzcd}[row sep=tiny, every label/.append style = {font = \normalsize}]
& &   \usebox\youngC \arrow[dash, dr, "s_4"] & & \\
 \usebox\youngA \arrow[dash, r, "s_3"] &  \usebox\youngB \arrow[dash, ur, "s_2"] \arrow[dash, dr, "s_4"] &  &  \usebox\youngE \arrow[dash, r, "s_3"] &  \usebox\youngF\\
& &  \usebox\youngD \arrow[dash, ur, "s_2"] & &
\end{tikzcd}$$
} 
    \caption{The $6$ elements of $T_{(3,1,1)}$. The edges indicate the simple transpositions that transform one standard Young tableau to another one.}
    \label{Figure2}
\end{figure}
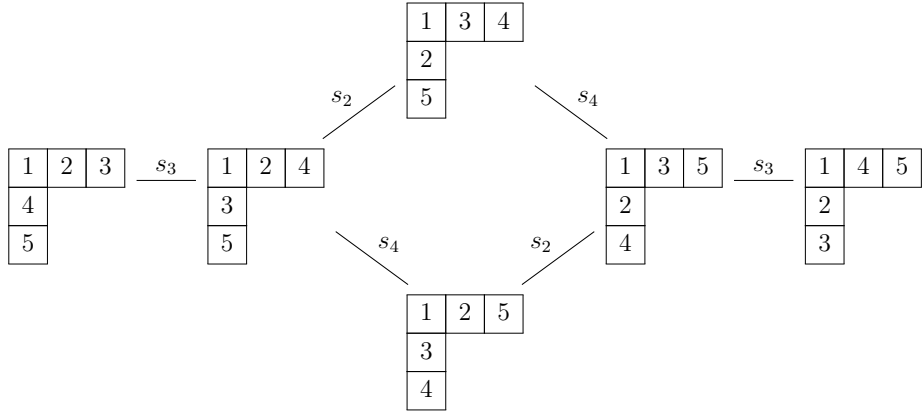

\begin{example}
\label{311Example}
Let $\lambda=(3,1,1)$ and let \[ t:= \adjustbox{scale=0.8}{  \usebox\youngB}\] be a standard Young tableau of $\lambda$. We wish to calculate $a_t$. It holds that $s_3 \cdot t_{\lambda}=t$, where the entry $3$ lies in position $(1,3)$ and the entry $4$ lies in position $(2,1)$ in $t_{\lambda}$, see Figure \ref{Figure2}. So $c=(3-1)-(1-2)-1=2$. By definition we now get \[a_t=x[4]_x[2]_x=x(x+1)^2(x^2+1).\]
Further easy calculations give us the result $$\det(\widetilde{\chi}_{\lambda}^{(q)})=(q^4+q^3+q^2+q+1) \cdot (\mathbb{Q}^{\times})^2.$$
\end{example}

\section{Orthogonal Determinants of  \texorpdfstring{$\mathrm{GL}_n(q)$}{GLn(q)}}
\label{OrthDetGLn}
Let $p$ be an odd prime and $q$ be a power of $p$. Let $n$ be a positive integer. In this chapter, we apply the statements of the previous chapters and achieve a complete description on the orthogonal determinants of the characters of $G:=\mathrm{GL}_n(q)$. 

The character tables of the general linear groups are fully known by the results of Green in \cite{GreenGeneralLinear}. The irreducible characters can be labeled by certain partition-valued functions, which allow to easily calculate the corresponding character fields and degrees, see for instance \cite[Proposition 2.8]{TurullSchurIndices} and \cite[Section 3.3.5]{HissFiniteGroupsLieType}. The Schur index of any irreducible character of $G$ is equal to $1$, see \cite[Proposition 12.6]{ZelevinskyHopfAlgebraGLn}. In short, we arrive that it is easy to determine whether any given irreducible character is an element of $\mathrm{Irr}^+(G)$, given just its combinatorial description. The general linear groups are examples of the finite groups of Lie type, see for instance \cite{digne_michel_2020} for more information.

We introduce some important structural subgroups of $G$:
\begin{enumerate}[label=(\roman*)]
    \item Let again $B \subseteq G$ be the Borel subgroup of upper triangular matrices. Note that $B$ is solvable.
    \item Let $T \subseteq B$ be the subgroup of diagonal matrices. Note that $T$ is abelian.
    \item Let $U \subseteq B$ be the normal subgroup of upper unitriangular matrices, i.e., the upper triangular matrices with only $1$s on the diagonal. Note that $U$ is a $p$-Sylow subgroup.
\end{enumerate}
The Borel subgroup can be decomposed as a semidirect product $B=T \ltimes U$. The Weyl group of $G$ is equal to $\mathfrak{S}_n$, with an action of $\mathfrak{S}_n$ on $T$ given by the permutation of diagonal entries. Moreover, this induces an action of $\mathfrak{S}_n$ on $\mathrm{Irr}(T)$.

The following notation was introduced in \cite[Remark 4.1]{HoyerNebe}. Let $\chi \in \mathrm{Irr}(G)$.
\begin{definition}
\label{DefinitionChiTChiU}
 Let $\chi \in \mathrm{Irr}(G)$. Let $\chi_T$ be the largest constituent of $\mathrm{Res}^G_B(\chi)$ such that its restriction to $U$ is trivial, i.e., \[ \mathrm{Res}^B_U(\chi_T)= \langle \mathrm{Res}^G_U(\chi),\mathbf{1}_U \rangle \cdot \mathbf{1}_U.  \]  We decompose $\mathrm{Res}^G_B(\chi)=\chi_T+\chi_U$, where we may regard $\chi_T$ (resp. $\chi_U$) as a character of $T$ (resp. $U$).  
\end{definition}
The character $\chi_T$ is also known as the \textit{Harish-Chandra restriction} of $\chi$. Note that Definition \ref{DefinitionChiTChiU} is a character-theoretic version of Definition \ref{StableModule}.
\begin{proposition}(see \cite[Theorem 5.3.7]{digne_michel_2020}
\label{ChiTProposition}
Assume that $\chi_T$ is non-zero; we also say that $\chi$ lies in the \textit{principal series}. Let $\theta \in \mathrm{Irr}(T)$ be a constituent of $\chi_T$. Then \[ \chi_T=a  \sum_{w \in \mathfrak{S}_n} w\cdot \theta \] for some $a \in \mathbb{Q}$. 
\end{proposition}
Assume from now on that $\chi \in \mathrm{Irr}^+(G)$. 
\begin{proposition}
\label{ChiUDet}
 The character $\chi_U$ is orthogonally stable with $\mathbb{Q}(\chi_U)=\mathbb{Q}$. Moreover, $q-1 \mid \chi_U(1)$ and so \[ \det(\chi_U)=q^{\chi_U(1)/(q-1)} \cdot (\mathbb{Q}^{\times})^2.\]   
\end{proposition}
\begin{proof}
The first statement follows since $\chi_U$ does not contain the trivial character by definition and by $U$ being a $p$-group for an odd prime $p$. Moreover, in \cite[Theorem 1.9]{TIEPPaper} it is shown that $\mathbb{Q}(\chi_U)=\mathbb{Q}$. The fact that $q-1 \mid \chi_U(1)$ can be deduced from the explicit formulas for the character degrees of $\chi$ and $\chi_T$, so we can now apply Remark \ref{Abelian Group p-Group Determinant}.   
\end{proof}

\begin{corollary}
Assume $\chi$ does not lie in the principal series. Then \[ \det(\chi)= q^{\chi_U(1)/(q-1)} \cdot (\mathbb{Q}(\chi)^{\times})^2. \]    
\end{corollary}
\begin{proof}
This follows since $\mathrm{Res}^G_U(\chi)=\chi_U$ and by Lemma \ref{OrthDetRestriction}.    
\end{proof}
From now on, assume that $\chi$ belongs to the principal series. Fix a constituent $\theta \in \mathrm{Irr}(T)$ of $\chi_T$. Since $\mathrm{Irr}(T) = \mathrm{Hom}(T, \mathbb{C}^\times)$ forms a group and by Proposition \ref{ChiTProposition}, the order of $\theta$ is well-defined and does not depend on the choice of $\theta$.  We will now examine all possible cases. Fix a generator $\varepsilon$ of $\mathbb{F}_q^\times$.

\subsection*{Case 1: $|\theta| \geq 3$}
In that case the Frobenius--Schur indicator of $\theta$ is equal to $0$. By Remark \ref{Abelian Group p-Group Determinant} $\chi_T$ is orthogonally stable. It follows that $\mathrm{Res}^G_B(\chi)$ is orthogonally stable, we also say that $\chi$ is \textit{Borel-stable}. Since $T$ is an abelian group, $\det(\chi_T)$ can be easily calculated. We have that \[ \det(\chi)=\det(\chi_T) q^{\chi_U(1)/(q-1)} \cdot (\mathbb{Q}(\chi)^{\times})^2 \] by Lemma \ref{OrthDetRestriction}. We refer to \cite{HoyerNebe} for a detailed discussion of the orthogonal determinants of Borel-stable characters for the groups $\mathrm{SL}_3(q)$ and $\mathrm{SU}_3(q)$.

\begin{example}
  Let $G=\mathrm{GL}_3(q)$. Let $\mu$ be a primitive complex $(q-1)$-th root of unity. We have that \[ T = \{ \mathrm{diag}(\varepsilon^{k}, \varepsilon^{\ell}, \varepsilon^{m}) \mid 0 \leq k,\ell,m \leq q-2 \}. \] Define the character $\theta \in \mathrm{Irr}(T)$ by \[  \theta \left( \mathrm{diag}(\varepsilon^{k}, \varepsilon^{\ell}, \varepsilon^{m}) \right):= \mu^{k-\ell}. \] Let $\chi:=\mathrm{Ind}^G_B(\theta)$; it holds that $\chi \in \mathrm{Irr}^+(G)$ with $\chi(1)=(q+1)(q^2+q+1)$. Moreover, $\mathbb{Q}(\chi)=\mathbb{Q}(\mu+\mu^{-1}).$

  It holds that \[ \chi_T=\sum_{w \in \mathfrak{S}_3} w \cdot \theta\] and so $\chi_T(1)=6.$ It is easy to see that \[ \det(\chi_T)=\det(\theta+\overline{\theta})=(\mu^2+\mu^{-2}) \cdot  (\mathbb{Q}(\mu+\mu^{-1})^{\times})^2, \] where we used Remark \ref{DetOrthogonallySimple} for the last equality. Since both $\chi(1)$ and $\chi_T(1)$ are known, we can deduce $\chi_U(1)$ and so $\det(\chi_U)=q \cdot (\mathbb{Q}^{\times})^2$ by Proposition \ref{ChiUDet}. Putting everything together, we conclude that \[ \det(\chi)=q  (\mu^2+\mu^{-2}) \cdot  (\mathbb{Q}(\mu+\mu^{-1})^{\times})^2.  \]
\end{example}
    \subsection*{Case 2: $|\theta|=1$} This is equivalent to $\theta=\mathbf{1}_T$. So by Definition \ref{DefinitionUnipotent}, $\chi=\chi_{\lambda}^{(q)}$ is a unipotent character for some partition $\lambda \vdash n$. Combining Theorem \ref{Condensation_OrthDet} and Theorem \ref{OrthogonalDeterminantsHeckeTypeAn}, we achieve the following:

\begin{theorem}
\label{Determinant Unipotent}
It holds that \[ \det( \chi)=f_{\lambda}(q) q^{\chi_U(1)/(q-1)} \cdot (\mathbb{Q}^{\times})^2.\]    
\end{theorem}
\begin{example}
Let $\lambda=(3,1,1)$. It holds that $\chi(1)=q^3(q^2+q+1)(q^2+1)$. Since $\chi_T(1)=6$, we can calculate $\chi_U(1)$ and get $\det(\chi_U)=1 \cdot (\mathbb{Q}^{\times})^2$. We have already calculated $f_\lambda(q)\cdot (\mathbb{Q}^{\times})^2$ in Example \ref{311Example}, so we conclude \[  \det(\chi)=(q^4+q^3+q^2+q+1) \cdot (\mathbb{Q}^{\times})^2 .\]
\end{example}

\subsection*{Case 3: $|\theta|=2$}
We describe an explicit construction of the character $\chi$, which follows from the combinatorial description of the characters of $G$. Define the character $\mathbf{sgn} \in \mathrm{Irr}(\mathbb{F}_q^{\times})$ by setting $\mathbf{sgn}(\varepsilon)=-1$. Let $k$ be a positive integer. The character $\mathbf{sgn}$ extends to a linear character $\mathbf{sgn}_k \in \mathrm{Irr}(\mathrm{GL}_k(q))$ by setting $\mathbf{sgn}_k:=\mathbf{sgn} \circ \det.$

\begin{definition}
Let $(\lambda, \mu)$ be a pair of partitions with $\lambda \vdash \ell$, $\mu \vdash m$ and $\ell+m=n$ for some non-negative integers $\ell,m$. Let $P \subseteq G$ be a parabolic subgroup with Levi subgroup $\mathrm{GL}_{\ell}(q) \times \mathrm{GL}_m(q)$. Define \[ \chi_{(\lambda, \mu)}:=\mathrm{Ind}^G_P \left( \chi_{\lambda}^{(q)} \boxtimes (\mathbf{sgn}_m \cdot \chi_{\mu}^{(q)})   \right) \in \mathrm{Irr}(G). \]
The characters $\chi_{(\lambda, \mu)}$, as we go over all such pairs of partitions, are precisely the irreducible characters of $G$ such that $(\chi_{(\lambda, \mu)})_T$ is a sum of linear characters of order $2$.
\end{definition}
So $\chi=\chi_{(\lambda,\mu)}$ for two such partitions $\lambda$ and $\mu$. Note that if $\chi_{\mu}^{(q)} \in \mathrm{Irr}^+(\mathrm{GL}_m(q))$, it clearly holds that \[ \det (\mathbf{sgn}_m \cdot \chi_{\mu}^{(q)}) =\det(\chi_{\mu}^{(q)}).\] Finally, $\det(\chi)$ can be deduced by Lemma \ref{Determinant Induced Character}, Lemma \ref{Determinants Direct Product Groups} and Theorem \ref{Determinant Unipotent}. 

\begin{remark}
The question arises whether the results for the general linear group $\mathrm{GL}_n(q)$ can be extended to the special linear groups $\mathrm{SL}_n(q)$ for $q$ odd. Indeed, almost everything carries over with almost no changes. For $q \equiv 3 \mod 4$ in particular, all orthogonal determinants can be deduced with the methods described in this section. 

Assume that $q \equiv 1 \mod 4$ and that $n=2k$ is even. For $n \geq 6$, there are characters in $\mathrm{Irr}^+(\mathrm{SL}_n(q))$ for which we need some new methods:

Let $\lambda$ be a partition of $k$. Then \[  \mathrm{Res}^G_{\mathrm{SL}_n(q)}(\chi_{(\lambda,\lambda)})=\chi_{(\lambda,+)}+ \chi_{(\lambda,-)} \] for some orthogonal characters $\chi_{(\lambda,+)}, \chi_{(\lambda,-)} \in \mathrm{Irr}(\mathrm{SL}_n(q))$, see \cite{TurullSchurIndices}. If $\chi_{(\lambda,+)} \in \mathrm{Irr}^+(\mathrm{SL}_n(q))$, Theorem \ref{Condensation_OrthDet} allows again to reduce the calculation to a monomial algebra $\mathcal{H}$. This algebra $\mathcal{H}$ is a deformation of a Coxeter group of type $D_n$, though its Schur basis and involution differs from the one of the corresponding Iwahori--Hecke algebra. Further investigation into the representation theory of $\mathcal{H}$ is required to compute its orthogonal determinants.
\end{remark}

\section{On a conjecture by Richard Parker}
We begin this section with some terminology.
\begin{definition}
Let $a$ be a positive rational number. We say that the square class $a \cdot (\mathbb{Q}^{\times})^2$ is \textit{odd} (resp. \textit{even}) if and only if the unique squarefree positive integer $a'$ such that \[a \cdot (\mathbb{Q}^{\times})^2 = a' \cdot (\mathbb{Q}^{\times})^2 \] is odd (resp. even).
\end{definition}

Having had calculated thousands of orthogonal determinants of finite (simple) groups, Richard Parker made the following conjecture.
\begin{conjecture}(see \cite[Conjecture 1.3]{NebeOrthDisc})
\label{ParkersConjecture}
Let $G$ be a finite group and let $\chi$ be an orthogonally stable character with $\mathbb{Q}(\chi)=\mathbb{Q}$. Then $\det(\chi)$ is odd.
\end{conjecture}
The paper \cite{NebeOrthDisc} contains a generalized version of this conjecture for arbitrary number fields and a proof of the conjecture for $G$ a solvable group. Moreover, the conjecture holds for the symmetric groups, see \cite{HoyerSymmetricParker}.

Let $n$ be a positive integer and let $q$ be a power of an odd prime. In this final section, we confirm the abovementioned (generalized) conjecture for the group $G:=\mathrm{GL}_n(q)$. First we note that it is enough to confirm it just for the $\mathrm{Irr}^+(G)$ characters. We now fix a character $\chi \in \mathrm{Irr}^+(G)$. By Section \ref{OrthDetGLn}, there are the following mutually exclusive cases:
\begin{enumerate}[label=(\roman*)]
\item $\chi$ is Borel-stable. Since any Borel subgroup is a solvable group, Conjecture \ref{ParkersConjecture} holds for $\chi$.
\item $\chi$ is unipotent.
\item $\chi_T$ is a sum of linear characters of degree $2$. Then either $\det(\chi)=1 \cdot (\mathbb{Q}^{\times})^2$ or $\det(\chi)=\det(\chi_{\lambda}^{(q)})$ for $\lambda \vdash k$ a partition of some positive integer $k \leq n$.
\end{enumerate}
We conclude that Conjecture \ref{ParkersConjecture} holds for $G$ if and only if it holds for any orthogonally stable unipotent character.

 \begin{definition}
The cyclotomic polynomials $\Phi_n(x)$ for positive integers $n$ are inductively defined by the condition \[ x^n-1=\prod_{d \mid n}\Phi_d(x).\]
 \end{definition}
 In particular, \[ [n]_x=\frac{\prod_{d \mid n}\Phi_d(x)}{x-1}.\] The following gives some very basic properties of cyclotomic polynomials, see \cite[§46, §48]{NagellNumberTheory}.
\begin{lemma}
\label{CyclotomicLemma}
  Let $n$ be a positive integer.
    \begin{enumerate}[label=\roman*)]
        \item $\Phi_n(x) \in \mathbb Z[x]$.
        \item $\Phi_{2^n}(x)=x^{2^{n-1}}+1$.
        \item $\Phi_n(1)=\begin{cases}
            s, &\text{if} \ n=s^k \ \text{for some prime} \ s, \\
            0, &\text{if} \ n=1, \\
            1, &\text{else}. 
        \end{cases}$
    \end{enumerate}
\end{lemma}

\begin{lemma}
\label{TechnicalLemmaCyclotomic}
Let $c$ be a positive integer. Then the square classes $c(c+2) \cdot (\mathbb Q^{\times})^2$ and $[c]_q[c+2]_q \cdot (\mathbb Q^{\times})^2$ have the same parity.
\end{lemma}
\begin{proof}
Since $q$ is odd, for any polynomial $f \in \mathbb{Z}[x]$ it holds that $$f(1) \equiv f(q) \mod 2.$$ Let $f:=[c]_x[c+2]_x$. If $c$ is odd, then clearly both $f(1)=c(c+2)$ and $f(q)=[c]_q[c+2]_q$ are odd integers, so the statement holds here.

Assume that $c$ is even. Then exactly one of $c/2$ and $(c+2)/2$ is odd. Let $2^k$ be the biggest power of $2$ such that $c(c+2)=2 \cdot 2^k \cdot m$ for $m$ an odd integer. Clearly $k \geq 2$. We write
\begin{equation}
\label{Equation1}
f(x)=[c]_x[c+2]_x=(x+1)^2 \left( \prod_{i=2}^k \Phi_{2^i}(x) \right) \cdot g(x),
\end{equation}

where $g \in \mathbb{Z}[x]$ is the product of all cyclotomic polynomials $\Phi_d(x)$ for the divisors $d$ of either $c$ or $c+2$ that are not powers of $2$.

We go through the three factors of $f$ on the right hand side of \ref{Equation1}. The term $(x+1)^2$ can be disregarded since we consider square classes. By Lemma \ref{CyclotomicLemma}(ii) we observe that for any $i \geq 2$ it holds that $\Phi_{2^i}(q)=q^{2^{i-1}}+1=2 \cdot r$ for some odd integer $r$. So, clearly \[ \prod_{i=2}^k \Phi_{2^i}(1) \cdot (\mathbb{Q}^{\times})^2=2^{k-1} \cdot (\mathbb{Q}^{\times})^2 \] and \[\prod_{i=2}^k \Phi_{2^i}(q) \cdot (\mathbb{Q}^{\times})^2 \] have the same parity. Finally, by Lemma \ref{CyclotomicLemma}(iii), it holds that $g(1)$ is an odd integer, and therefore $g(q)$ is also an odd integer. This concludes the proof.
\end{proof}

\begin{proposition}
Let $\lambda \vdash n$ such that the unipotent character $\chi_{\lambda}^{(q)} \in \mathrm{Irr}^+(G)$. Then $\det(\chi)$ is odd.
\end{proposition}
\begin{proof}
The idea is to show that $\det( \chi_{\lambda}^{(q)})$ is odd if and only if $\det( \chi_{\lambda}^{(1)})$ is odd. Since Conjecture \ref{ParkersConjecture} holds for the symmetric groups by \cite{HoyerSymmetricParker}, the statement then follows.

Recall by Theorem \ref{Determinant Unipotent} and Theorem \ref{OrthogonalDeterminantsHeckeTypeAn} that \[ \det( \chi_{\lambda}^{(q)})=f_{\lambda}(q) q^{\chi_U(1)/(q-1)} \cdot (\mathbb{Q}^{\times})^2\]  and \[ \det( \chi_{\lambda}^{(1)})=f_{\lambda}(1) \cdot (\mathbb{Q}^{\times})^2.\]

So it suffices to show that $f_{\lambda}(q) \cdot (\mathbb{Q}^{\times})^2$ and $f_{\lambda}(1) \cdot (\mathbb{Q}^{\times})^2$ have the same parity. Recall that $f_{\lambda}=\prod_{t \in T_{\lambda}} a_t$. By the definition of the polynomials $a_t$, we can further reduce to showing that $q[c+2]_q[c]_q \cdot (\mathbb{Q}^{\times})^2$ and $(c+2)c \cdot (\mathbb{Q}^{\times})^2$ have the same parity, for any positive integer $c$. But this follows from Lemma \ref{TechnicalLemmaCyclotomic} and so we are done.
\end{proof}

To summarize, we have shown the following.
\begin{theorem}
Conjecture \ref{ParkersConjecture} holds for the groups $\mathrm{GL}_n(q)$ for $q$ a power of an odd prime and $n$ a positive integer.
\end{theorem}

\bibliography{References}
\end{document}